\documentclass[12pt]{amsart}
\usepackage{epsfig,color}

\makeatother

\theoremstyle{definition}

\def\fnum{equation} 
\newtheorem{Thm}[\fnum]{Theorem}
\newtheorem{Cor}[\fnum]{Corollary}

\newtheorem{Lem}[\fnum]{Lemma}

\numberwithin{equation}{section}

\newcommand{\Vol}{{\text{Vol}}}

\def\ZZ{{\bold Z}}
\def\RR{{\bold R}}

\newcommand{\dv}{{\text {div}}}
\newcommand{\e}{{\text {e}}}

\newcommand{\cL}{{\mathcal{L}}}

\newcommand{\eqr}[1]{(\ref{#1})}

\begin{document}

\title[Frequency bounds for Ornstein-Uhlenbeck eigenfunctions]{Sharp frequency bounds for eigenfunctions of the Ornstein-Uhlenbeck operator}
\author[]{Tobias Holck Colding}%
\address{MIT, Dept. of Math.\\
77 Massachusetts Avenue, Cambridge, MA 02139-4307.}
\author[]{William P. Minicozzi II}%
 
\thanks{The  authors
were partially supported by NSF Grants DMS 1404540,  DMS 1206827 and DMS 1707270.}


\email{colding@math.mit.edu  and minicozz@math.mit.edu}

\maketitle

\begin{abstract}
We prove sharp bounds for the  growth rate of eigenfunctions of the Ornstein-Uhlenbeck operator and its natural generalizations.  The bounds are sharp even up to lower order terms and have important applications to geometric flows.  
\end{abstract}

\section{Introduction}

The Ornstein-Uhlenbeck operator (or drift Laplacian),  $\cL$ on $\RR^n$ is the second order operator $\cL u = \Delta u - \langle \nabla f , \nabla u \rangle$, where $f = \frac{|x|^2}{4}$.  It is self-adjoint with respect to the Gaussian $L^2$ inner product whose norm is $\| u \|_{L^2}^2 = \int u^2 \, \e^{-f}$.   
We study here the   rate of growth of drift eigenfunctions $u$ with $\cL\,u=- \lambda u$.  
The  results given here  are important ingredients in the proof of the Ren\'e Thom gradient conjecture for the arrival time function; see \cite{CM3}.

 \vskip2mm
 It is easy to see that if $\cL u =0$ and $\| u \|_{L^2} < \infty$, then $u$ must be constant.  More generally, if $\cL u = - \lambda u$ and $\| u \|_{L^2} < \infty$, then $\lambda$ is a half-integer and $u$ is a polynomial of degree $2\lambda$.  When $n=1$, these polynomials are the Hermite polynomials and the equation $\cL u = - \lambda u$ is Hermite's equation.    Hermite's equation has a dichotomy where either a solution is polynomial, or it grows faster than any exponential.

We will consider a more general class of drift Schr\"odinger  equations, where 
  $u$ satisfies  
\begin{align}
\cL_f\,u+V\,u = 0 \, ,
\end{align}
for some function $V$ and some function $f(x)=f(|x|)$, where $f$ only depends on the distance to the origin.  For
 the Ornstein-Uhlenbeck operator,  $f(r)=\frac{r^2}{4}$ and, thus, $f'(r)=\frac{r}{2}$.  

The  frequency $U$ of $u$   measures  the   rate of growth of $u$.  If $u(x) = |x|^d$, then $U= d$. We show:  

\begin{Thm}   \label{t:main}
Suppose that $f'(r)\geq \frac{r}{2}$.  Given $\epsilon>0$ and $\delta>0$, there exist $r_1>0$ such that if $U(\bar{r}_1)\geq \delta+2\,\sup \, \{ 0 ,   V \} $ for some $\bar{r}_1 \geq r_1$, then for all $r\geq R(\bar{r}_1)$
\begin{align}     
U(r) &> \frac{r^2}{2}-n-2\,\sup V-\epsilon\, . \label{e:inductive}  
\end{align}
\end{Thm}

We will construct examples  that show that the lower bound for $U$ is sharp in all dimensions; not only is the quadratic coefficient $\frac{1}{2}$ sharp, but also the constant $-n$ cannot be improved.  

The theorem is also sharp in the dependence on the $\sup V$.  Namely,  if $V= \frac{k}{2}$ is a positive half-integer, then the polynomial solutions mentioned above 
have $U$ asymptotic to $k$.  Thus, the threshold $\delta + 2\sup V$ is sharp.  Furthermore, we will see that \eqr{e:inductive}  
 is also sharp in $\sup V$.
 
 Theorem \ref{t:main} shows that there is a sharp dichotomy for the growth:  either $U$ is bounded and $u$ grows at most polynomially, or $u$ grows at least like $r^{-n-2\,\sup V} \,\e^{ \frac{r^2}{4}}$.

\vskip2mm
For eigenfunctions of the Ornstein-Uhlenbeck operator, where $f(r)=\frac{r^2}{4}$, we also get a lower bound for the derivative of the frequency:

 \begin{Thm}  \label{t:mainLem2}
 If  $f(x)=\frac{|x|^2}{4}$ and $\cL\,u+\lambda\,u=0$, then either $\limsup_{r \to \infty} U(r) \leq 2 |\lambda|$ or there exists $R$ so that for all $r \geq R$
 \begin{align}	\label{e:0p5}
 U'\geq  \frac{r}{2}\left(1+\frac{r^2}{2\,n+4+4\,U(r)-r^2}- \frac{2\,n+8\,\lambda}{2n + 4\, U(r) - r^2} \right)  +   \frac{O(r^{1-n})}{{2n + 4\, U(r) - r^2}} \, , 
 \end{align}
 where $O(r^{1-n})$ is a term that is bounded by a constant  times $r^{1-n}$.
 \end{Thm}
 
If we set $W=U-\frac{r^2}{4}+\frac{n}{2}$, then \eqr{e:0p5}  becomes 
$
 W'\geq \frac{r}{8}\,\left(\frac{r^2}{W+1}-\frac{2\,n+8\,\lambda}{W}\right)
$ up to lower order terms.  Integrating leads to the bound $U \geq \frac{1}{2} \, r^2 - n-1-2\,\lambda$, which is slightly worse than \eqr{e:inductive}.  
However, this inequality gives a (positive) derivative bound for all values of $U$.

\vskip2mm
Our arguments are quite flexible and generalize.  For instance:

\begin{Thm}  \label{t:openintro}
Suppose that $f'(r)\geq \frac{r}{2}$.  Let $M$ be an open manifold with nonnegative Ricci curvature, Euclidean volume growth and Green's function $G$.  Fix $x_0\in M$ and let $b$ be given by $b^{2-n}=G(x_0,\cdot)$. Given $\epsilon>0$ and $\delta>0$, there exist $r_1>0$  such that if  $\cL_{f(b)}\,u=0$ and 
$U(\bar{r})\geq \delta$ for some $\bar{r} \geq r_1$, then for all $r\geq R(\bar{r})$
\begin{align}     
U(r) &> \frac{r^2}{2}-n-\epsilon\, .
\end{align} 
\end{Thm}

\vskip2mm
In this theorem, $\cL_{f(b)}\,u=\Delta-\langle \nabla u,\nabla f(b)\rangle$ and $I$, $D$, and $U$ are defined in terms of $b$; see \eqr{e:Ib}, \eqr{e:Db} and \eqr{e:Ub}.

  \section{The sharp lower bound for $U$}
   
   In this section, $f:\RR^n\to \RR$ is a function that only depends on the distance to the origin.  With slight abuse of notation we write $f(x)=f(|x|)$ and denote $\partial_rf$ by $f'$.  
   
   \vskip1mm
Define quantities $I(r)$, $D(r)$, and the frequency $U(r)$ by
\begin{align}
	I (r) &= r^{1-n} \int_{\partial B_r} u^2 \, , 	\label{e:defI}\\
D(r)& =  r^{2-n} \int_{\partial B_r} u u_r  =r^{2-n} \, \e^{  f(r)} \, \int_{B_r}  \left(|\nabla u|^2-V\,u^2\right) \, \e^{-f}  \, ,\label{e:defD}\\
U(r)&= \frac{D}{I}\, .\label{e:defU}
\end{align}
The frequency $U$ is   the logarithmic derivative of    $\frac{1}{2} \log I$, i.e., $(\log I)' = \frac{2U}{r}$, and thus  measures the polynomial rate of growth of $\sqrt{I}$.
This frequency was  recently used by Bernstein, \cite{B}, to study the asymptotic structure of ends of shrinkers for mean curvature flow.
It is
analogous to a similar quantity for harmonic functions known as Almgren's frequency function, \cite{A}, cf. \cite{GL}, \cite{HS}, \cite{L}, \cite{CM1}, \cite{D}.

  An easy calculation together with that $\dv \left( \e^{-f} \, \nabla v \right) = \e^{-f} \, \cL_f\, v$ shows 
 \begin{align}
 \frac{d}{dr}\left(r^{1-n}\int_{\partial B_r} v \right) &= r^{1-n}\int_{\partial B_r}\frac{dv}{dr} =r^{1-n}\left(\e^{f(r)}\int_{B_r}\cL_f\,v\,\e^{-f}\right)\, . \label{e:diffI}
 \end{align}
Using \eqr{e:diffI} with $\cL_f\, u = 0$ gives that
  the spherical average of a $\cL_f$-harmonic function is constant in $r$.

 \begin{Lem}	\label{l:IandD}
 If $\cL_f\,u+V\,u=0$, then
 \begin{align}
 I'(r) &= \frac{2D(r)}{r}\, ,\\
 ( \log I )' (r)&= \frac{2U(r)}{r}\, ,\\
D'(r)&=  \frac{2-n}{r} \, D + f'(r) \, D + r^{2-n} \,   \int_{\partial B_r}  \left( |\nabla u|^2 - V\, u^2 \right)   \, .\label{e:Donetime0}
 \end{align}
 \end{Lem}
 
 \begin{proof}
Since $\cL_f  \,u^2=2\,|\nabla u|^2-2\,V\,u^2$,  \eqr{e:diffI} gives
 \begin{align}
 	I'(r) = 2 r^{1-n} \, \e^{  f(r)} \, \int_{B_r} \left(|\nabla u|^2-V\,u^2\right) \, \e^{-f}= \frac{2D(r)}{r} \, .
 \end{align}
 This gives the first two claims.   Differentiating \eqr{e:defD} gives \eqr{e:Donetime0}.
 \end{proof}
 
Define   a (non-linear) first order differential operator on positive functions $g$ on $(0,\infty)$ by
 \begin{align}
 P_{f,\lambda}\,g=(\log g)'+\frac{n-2}{r}-f'+\frac{g}{r}+\frac{r\,\lambda}{g} \, .
\end{align}
We will later use that if $f_2'\geq f_1'$, then $P_{f_1,\lambda}\,g\geq P_{f_2,\lambda}\,g$.   

The key will be that $U$ is a sub-solution of $P$:
 
\begin{Lem}  \label{l:Uonetime}
If $\cL_f\,u+V\,u=0$ and $\infty>U(r) >0$, then 
\begin{align}	
	P_{f, \,\sup V}\,U\geq 0 \, .\label{e:Uonetime}
\end{align}
\end{Lem} 

\begin{proof}
 The Cauchy-Schwarz inequality 
 \begin{align}	\label{e:CS}
 	\frac{D^2}{r} = r^{3-2n} \, \left(\int_{\partial B_r} u u_r \right)^2 \leq  I \, r^{2-n} \int_{\partial B_r} u_r^2 \leq I \, r^{2-n} \int_{\partial B_r} |\nabla u|^2 
\end{align}
together with \eqr{e:Donetime0} gives
 \begin{align}	 \label{e:Donetime}	
D'(r)\geq  \frac{2-n}{r} \, D + f'(r) \, D + \frac{U}{r} \, D - r\,\sup V \, \frac{D}{U} \, ,
 \end{align}
Since $(\log U)' = \frac{D'}{D} - \frac{2U}{r}$ and $D(r) > 0$, dividing \eqr{e:Donetime} by $D$ gives  \eqr{e:Uonetime}.  
\end{proof}

The next lemma shows a maximum principle for the operator $P_{f,\lambda}$.  

\begin{Lem}   	\label{l:indu}
Suppose that $g,\,h:\RR\to (0,\infty)$ satisfy for $r \geq r_1 $
\begin{align}
P_{f,\lambda}\,h\geq 0>P_{f,\lambda}\,g \, .
\end{align}
If $h(R)> g(R)$ for some $ R \geq r_1$, then $h(r) > g(r)$ for all $r\geq R$.

Moreover, if $\epsilon > 0\geq \lambda$,  and $g$ satisfies
\begin{align}
 -\frac{\epsilon}{r} \geq P_{f,\lambda}\,g{\text{ for }} r \geq r_1 \, ,	\label{e:assum2}
\end{align}
then there exists $R= R(h(r_1) , g(r_1) , r_1 ,   \epsilon)$ so that $h \geq g$ for $r \geq R$.
\end{Lem}

\begin{proof}
We will prove the first claim by contradiction.  Suppose not, then there exists $s>R$ such that $h(s)=g(s)$ and $h(t)>  g(t)$ for all $s>t\geq R$.  This implies that
\begin{align}
(\log h)'(s)\leq (\log g)'(s)=\frac{g'}{g}\, .
\end{align}
On the other hand, by assumption $P_{f,\lambda}\,h\geq 0$ and thus
\begin{align}
(\log h)'(s) \geq \frac{2-n}{s}  + f'(s)   - \frac{h(s)}{s}   -   s \, \frac{\lambda}{h(s)} = \frac{2-n}{s}  + f'(s)   - \frac{g(s)}{s}   -   s \, \frac{\lambda}{g(s)}\, .
\end{align}
Together these two inequalities gives that $P_{f,\lambda}\,g\geq 0$ which is the desired contradiction.

The second claim will follow from the first once we show that there is some $R\geq r_1$ so that $h> g$ for some $r$ with $R\geq r\geq r_1$.  To see this, we suppose that
$h \leq g$ for $r_1\leq r \leq R$ and then get an upper bound on $R$. On this interval, 
since $\lambda\leq 0$ we get that
\begin{align}
(\log h)'(s) - (\log g)'(s) \geq P_{f,\lambda}\,h-P_{f,\lambda}\,g \geq \frac{\epsilon}{s} \, .
\end{align}
Integrating this from $r_1$ to $R$ gives
\begin{align}
	 1\geq \frac{h(R)}{g(R)} \geq \frac{h(r_1)}{g(r_1)} \, \left( \frac{R}{r_1} \right)^{\epsilon} \, .
\end{align}
Thus, we see that $R^{\epsilon} \leq r_1^{\epsilon} \, \frac{g(r_1)}{h(r_1)}$.  
\end{proof}

\begin{Lem}	\label{l:chooseg}
Suppose that $f(r) = \frac{r^2}{4} $, $\epsilon > 0$ and let 
$g(r) = \frac{r^2}{2}-n-\epsilon-2\,\lambda$, 
then there exists $r_1=r_1(\epsilon,n)$  so that for $r \geq r_1$  
\begin{align}
-\frac{\epsilon}{2r} \geq P_{f,\lambda}\,g \, ,
\end{align}
\end{Lem}

\begin{proof}
Choose $r_1$ so that for $r\geq r_1$
\begin{align}  \label{e:choicer_1a}
\frac{2 \,(\lambda+1)}{1-2\,(n +\epsilon+2\,\lambda)/r^2} \leq 2 +2\,\lambda+\frac{1}{2} \, \epsilon\, .
\end{align}
For $r\geq r_1$, \eqr{e:choicer_1a} implies that
\begin{align}   \label{e:choicer_1b}
	\frac{2\, r\, (\lambda+1) }{ r^2-2\,(n+\epsilon+2\,\lambda)} = \frac{1}{r} \left( \frac{2\,(\lambda+1) }{1-2\,(n +\epsilon+2\,\lambda)/r^2} \right)   \leq \frac{2 +2\,\lambda+\frac{1}{2} \, \epsilon}{r} \, .
\end{align}
Using the definitions of $f$ and $g$, we get
\begin{align}
-P_{f,\lambda}\,g&=\frac{2-n}{r}+f'-\frac{g}{r}-\frac{r\,\lambda}{g} - \frac{g'}{g}= \frac{2 +\epsilon+2\,\lambda }{r}-\frac{2\,r\,(\lambda+1)}{r^2-2\,(n+\epsilon+2\,\lambda)} \geq  
 \frac{ \epsilon }{2r}  \, .
\end{align}
\end{proof}

Combining the two previous results and Lemma \ref{l:Uonetime}  (to see that $P_{f,\sup V}\,U\geq 0$) gives Theorem \ref{t:main} in the case where $\lambda \leq 0$.  The argument for a general $\lambda$ is similar but a little more involved since
we need a replacement for the second half of Lemma \ref{l:indu}.  We will deal with this in the next subsection.

\subsection{The case $\lambda > 0$}

The next lemma will replace the second half of Lemma \ref{l:indu} when $\lambda > 0$.   

\begin{Lem}   \label{l:generalbound}
Suppose that    $\lambda > 0$ and for $r\geq r_1$ we have that $g,h > 0$, $f'\geq \frac{r}{2}$, $P_{f,\lambda}\,h\geq 0$, $-\frac{\epsilon}{r}\geq P_{f,\lambda}\,g$, and $r\,g'\geq  g$.  If $r_2 \geq r_1$ satisfies $\frac{2-n}{r_2}  + \frac{r_2}{2}      -   r_2 \, \frac{\lambda}{\delta+2\,\lambda} >  \sqrt{\lambda}$ and $ h( r_2)> 2\,\lambda+\delta$, then there exists $R$ such that $h(r)\geq g(r)$ for $r\geq R$.  
\end{Lem}

\begin{proof}
First, if $2\lambda + \delta \leq h(r)  < \sqrt{\lambda} \, r$ for $r \geq r_2$, then   $P_{f,\lambda}\,h\geq 0$  implies that
\begin{align}
(\log h)'(r) \geq \frac{2-n}{r}  + \frac{r}{2}   - \sqrt{\lambda}   -   r \, \frac{\lambda}{\delta+2\,\lambda} > 0  \, .
\end{align}

Second,  since $P_{f,\lambda}\, h\geq 0$ and $-\frac{\epsilon}{r}\geq P_{f,\lambda}\,g$, then
\begin{align}
\left(\log \frac{h}{g}\right)'=P_{f,\lambda}\,h-P_{f,\lambda}\,g+\frac{g-h}{r}-\lambda\,r\,\frac{g-h}{g\,h}\geq \frac{\epsilon}{r}+(g-h)\,\left(\frac{1}{r}-\frac{\lambda\,r}{g\,h}\right)\, .
\end{align}
Therefore, if $\sqrt{\lambda}\,r\leq h(r)< g(r)$, then
\begin{align}
\left(\log \frac{h}{g}\right)'(r)\geq \frac{\epsilon}{r}\, ,
\end{align}
and hence, using also that $rg' \geq g$ (this is the only place where this is used), we have
\begin{align}
(\log h)'(r)\geq \frac{\epsilon}{r}+(\log g)'(r)\geq \frac{1+\epsilon}{r}\, .
\end{align}
Thus, when $\sqrt{\lambda}\,r\leq h(r)< g(r)$, we have that
\begin{align}
(h-\sqrt{\lambda}\,r)'\geq \frac{(1+\epsilon)\,h(r)}{r}-\sqrt{\lambda}\geq (1+\epsilon)\,\sqrt{\lambda}-\sqrt{\lambda}=\epsilon\,\sqrt{\lambda}>0\, .
\end{align}

In both cases, we get that $h$ only leaves each bound at the upper end and we get an upper bound for the length of the stretch where $h$ has this bound.
Finally, it follows from the first part of Lemma \ref{l:indu} that once $h$ is above $g$ it stays above. 
\end{proof}

We can now get rid of the assumption that $r\,g'\geq g$ in Lemma \ref{l:generalbound} to get:

\begin{Thm}   \label{t:generalbound}
Suppose that $\lambda > 0$ and  for $r\geq r_1$ we have that $g,h > 0$,  $f'\geq \frac{r}{2}$, $P_{f,\lambda}\,h\geq 0$, $-\frac{\epsilon}{r}\geq P_{f,\lambda}\,g$, 
then there exists $r_2 > 0$ so that if 
  $  h (s)> 2\,\lambda+\delta$ for some $s \geq r_2$, then  there exists $R$ so that $h(r)\geq g(r)$ for $r\geq R $.  

In particular, $h(r)\geq \frac{r^2}{2}-n-2\,\lambda-\epsilon$ for $r\geq R$.
\end{Thm}

\begin{proof}
We show the second claim first and then use it to show the first claim.    To do that note that if $g_0=\frac{r^2}{2}-n-2\,\lambda-\epsilon$, then $r\,g_0'=r^2\geq g_0$.    Moreover,   Lemma \ref{l:chooseg} gives
$-\frac{\epsilon}{2r} \geq P_{f,\lambda}\,g_0$.  It follows  from Lemma \ref{l:generalbound} that for some $R>0$ and all $r>R$ we have that $h(r)\geq \frac{r^2}{2}-n-2\,\lambda-\epsilon$.

To show the first claim, note that in the proof of Lemma \ref{l:generalbound} the only place where the assumption $r\,g'\geq g$ was used was to show that there exists some $R$ so that once $r\geq R$ and $h(r)\geq \sqrt{\lambda}\,r$ the function $h$ would stay above the function $\sqrt{\lambda}\,r$.   However, this follows from  $h(r)\geq \frac{r^2}{2}-n-2\,\lambda-\epsilon$ for $r$ large enough.
\end{proof}

\begin{proof}[Proof of Theorem \ref{t:main}]
We have already proven the case $\lambda \leq 0$.  The case  $\lambda > 0$   follows from
  Lemma \ref{l:Uonetime}    and  Theorem \ref{t:generalbound}.   
 \end{proof}

 \subsection{Sharpness of  Theorem \ref{t:main}}
 
 The next theorem uses standard solutions of Hermite's equation to show  that
 Theorem \ref{t:main} is sharp even up to the lower order term.
 
 \begin{Thm}	\label{t:sharp}
  For every $n$ and $k \in \ZZ$, there is a  function $v$ on $\RR^n$ with $\cL v = - \frac{k}{2} \, v$ whose frequency $U$ goes to infinity but for every $\epsilon > 0$ has a sequence $r_i$ going to infinity with
 \begin{align}	\label{e:rik}
 	U(r_i) \leq \frac{1}{2} \, r_i^2 - n -k + \epsilon \, .
 \end{align}
 \end{Thm}
 
 The second order ODE $\cL\, u = 0$ on $\RR$, where $f(r)=\frac{r^2}{4}$, has a two-parameter family of solutions.  The first solution is a constant.  The  second, $u_0(x)$,  can be normalized to have $u_0(0)=0$ and $u_0'(0)=1$.  The next lemma shows that $I(r) \approx \frac{1}{r}\, \e^{ \frac{r^2}{4}}$.
  
 \begin{Lem}	\label{l:ODE}
 The function $u_0$ is odd,  $u_0'(x) = \e^{ \frac{x^2}{4} }$, and   for $ x \geq 2$  
 \begin{align}
 	  \e^{ \frac{x^2}{4} }     \leq x\, u_0(x) \leq   6\, \e^{ \frac{x^2}{4}} \, .	\label{e:u0}
 \end{align}
 Moreover, there are functions $u_k$ for all $k \in \ZZ$ with $\cL u_k = - \frac{k}{2} \, u_k$, so that $u_k' = u_{k-1}$  and, furthermore,  there are constants $c_k$ so that 
  \begin{align}
 	    |u(x)| \leq   c_k \, |x|^{-k-1} \, \e^{ \frac{x^2}{4}} {\text{ for }} 1 \leq |x| \, .	\label{e:uk}
 \end{align}
 \end{Lem}
 
 \begin{proof}
 Since $  \cL u_0 = u_0'' - \frac{x}{2} \, u_0' = \e^{ \frac{x^2}{4} } \, \left( u_0' \e^{-\frac{x^2}{4} } \right)'$, we see that $\left( u_0' \e^{-\frac{x^2}{4} } \right)$ is constant.  Using the normalization $u_0'(0)=1$,  the constant is one.  For the lower bound given $r>2$, we have
 \begin{align}
 	r\, u_0(r) = r \, \int_0^r \e^{ \frac{x^2}{4} } \, dx \geq \int_0^r x\, \e^{ \frac{x^2}{4} } \, dx = 2 \, \e^{ \frac{x^2}{4} }\big|_0^r = 2 \, \e^{ \frac{r^2}{4} } - 2 \geq  \e^{ \frac{r^2}{4} } \, .
 \end{align}
 To get the upper bound, we divide the integral into three parts
  \begin{align}
  	u_0(r)  &\leq  \int_0^1 \e^{ \frac{x^2}{4} } \, dx +  \int_1^{r/2} x\, \e^{ \frac{x^2}{4} } \, dx + 
 	   \frac{2}{r} \, \int_{r/2}^r x\, \e^{ \frac{x^2}{4} } \, dx \leq  \e^{ \frac{1}{4} } + 2 \, \left( \e^{ \frac{r^2}{16}} -  \e^{ \frac{1}{4} } \right) + \frac{4}{r} \e^{ \frac{r^2}{4}} \notag \\
	   &\leq 2 \,  \e^{ \frac{r^2}{16}}  + \frac{4}{r} \e^{ \frac{r^2}{4}}
	  \leq  \frac{6}{r} \e^{ \frac{r^2}{4}}  \, ,
 \end{align}
 where the last inequality used  $r^{-1} \,\e^{ \frac{3r^2}{16}}$ is increasing for $r \geq 2$ and $ \e^{ \frac{r^2}{16}} \leq \frac{1}{r} \e^{ \frac{r^2}{4}}$ at $r =2$.   
 
We construct the $u_k$'s for $k$ inductively for $k< 0$ by defining $u_k = u_{k+1}'$.   Using the bound \eqr{e:u0} and elliptic estimates on balls of radius $|x|^{-1}$ gives the bound \eqr{e:uk}.

 For $k\geq 0$, we inductively define
 \begin{align}
 	u_{k+1} (x) = \int_0^x u_k (s) \, ds + d_{k+1} \, , 
 \end{align}
 where the constant $d_{k+1}$ is chosen to make $\cL u_{k+1} = - \frac{k+1}{2} \, u_{k+1}$.      To see that 
 we can choose $d_{k+1}$ so that it satisfies the equation, note that
 \begin{align}
 	\left( \cL u_{k+1}  + \frac{k+1}{2} \, u_{k+1}\right)' &= \left( u_{k}' - \frac{x}{2} \, u_{k} + \frac{k+1}{2} \, u_{k+1}\right)' \notag \\
	&= u_{k}'' - \frac{x}{2} \, u_{k}' - \frac{1}{2} \, u_k + \frac{k+1}{2} \, u_{k} = 0 \, .
 \end{align}
  Using integration by parts, it is easy to see that $u_{k+1}$  grows one degree slower than $u_k$ and, thus, satisfies \eqr{e:uk}.    
  \end{proof}
  
  \begin{proof}[Proof of Theorem \ref{t:sharp}]
  It suffices to construct $v_k$ with $\cL v_k =- \frac{k}{2} \, v_k$ where $v_k$ grows at least exponentially  and has (for all $x$ sufficiently large)
  \begin{align}	\label{e:suffice}
  	|v_k| \leq C \, \e^{ \frac{|x|^2}{2} } \, |x|^{-k-n} \, .
  \end{align}
  This is because the failure of \eqr{e:rik} for all $r_i$ larger than some fixed $R$ implies  $\e^{ \frac{|x|^2}{2} } \, |x|^{\epsilon-k-n}$ growth $r \geq R$, contradicting \eqr{e:suffice}.
  
  The function $u_k$  from Lemma \ref{l:ODE} satisfies \eqr{e:suffice} for $n=1$.  For $n> 1$, we set
  \begin{align}	\label{e:product}
  	v_k(x_1 , \dots , x_n) = u_k(x_1) u_0(x_2) \dots u_0(x_n) \, .
  \end{align}
\end{proof}

  \section{Lower bound for $U'$}
  
In this section, we specialize to the Ornstein-Uhlenbeck operator $\cL$   where $f(r)=\frac{r^2}{4}$.     In this case,
$\cL\,f=\frac{n}{2}-f$,   $|\nabla f|^2=f$, and the Hessian of $f$ is diagonal with $f_{ij} = \frac{1}{2} \, \delta_{ij}$.
  
  \vskip1mm
  The next lemma is a drift version of the classical Rellich identity that is used to prove monotonicity  of Almgren's frequency for harmonic functions.
  
  \begin{Lem}	\label{l:poho}
Suppose that $\cL\,u+V\,u=0$ on $\RR^n$.   Given $r> 0$, we have
   \begin{align}	\label{e:poho}
 	2 \, r \int_{\partial B_r} u_r^2 - r\int_{\partial B_r}\left( |\nabla u|^2 - V\, u^2 \right) &=(2-n)\,\e^{\frac{r^2}{4}}\int_{B_r} |\nabla u|^2\,\e^{-f} +
	2 \, \e^{\frac{r^2}{4}} \, \int_{B_r} |\nabla u|^2 \, f\, \e^{-f} \notag \\
	&+2\, \e^{\frac{r^2}{4}} \, \int_{B_r}V\,u^2 \,\left( \frac{n}{2} -  f \right) \, \e^{-f} +2\,\e^{\frac{r^2}{4}}\int_{B_r}u^2\,\langle \nabla V,\nabla f \rangle\,\e^{-f}\, .
 \end{align}
  \end{Lem}
 
 \begin{proof}
Using that $f_{ij} = \frac{1}{2} \, \delta_{ij}$, the divergence of $\langle \nabla f , \nabla u \rangle \nabla u - \frac{1}{2} |\nabla u|^2 \, \nabla f$ is
 \begin{align}
    (f_i u_i  u_j - \frac{1}{2} u_i^2 f_j )_j &=  f_i u_i u_{jj} + f_i u_{ij} u_j + f_{ij} u_i u_j -\frac{1}{2} \,  u_i^2 f_{jj} -  u_{ij} u_i f_j \notag \\
    &=  f_i u_i u_{jj}   + \frac{2-n}{4} \,  u_i^2  \, .
 \end{align}
 In particular, since $\dv_f\, X\equiv  \e^f \, \dv \left( \e^{-f} \, X \right)=\dv\, X-\langle \nabla f,X\rangle$, we see that
 \begin{align}
 	\dv_f &\left( \langle \nabla f , \nabla u \rangle \nabla u - \frac{1}{2} |\nabla u|^2 \, \nabla f \right) =  \langle \nabla f , \nabla u \rangle \cL u + \frac{2-n}{4} |\nabla u|^2 + \frac{1}{2} \, |\nabla u|^2 \,f \notag\, ,
 \end{align}
where the equality also used that $|\nabla f|^2=f$.  The divergence theorem gives
 \begin{align}	 
 	2 \, r \int_{\partial B_r} u_r^2 - r\int_{\partial B_r} |\nabla u|^2 &=(2-n)\,\e^{\frac{r^2}{4}}\int_{B_r} |\nabla u|^2\,\e^{-f} +
	2 \, \e^{\frac{r^2}{4}} \, \int_{B_r} |\nabla u|^2 \, f\, \e^{-f} \notag \\
	&+4\,  \e^{\frac{r^2}{4}} \, \int_{B_r}  \langle \nabla f , \nabla u \rangle \cL u \, \e^{-f}\, .
 \end{align}
 The lemma follows from this and taking $\dv_f$ of $\frac{1}{2} V \, u^2 \, \nabla f$ to get
 \begin{align}	\label{e:usecLf}
 \int_{B_r} \langle \nabla f,\nabla u\rangle\,\cL\,u \,\e^{-f}&=-\frac{1}{2}\int_{B_r}V\,\langle \nabla f,\nabla u^2\rangle\,\e^{-f}\notag\\
 &=\frac{1}{2}\int_{B_r}V\,u^2\,\left(\frac{n}{2}-f\right)\,\e^{-f}-\frac{r}{4} \e^{ - \frac{r^2}{4} } \, \int_{\partial B_r}V\,u^2+ \frac{1}{2} \, \int_{B_r}u^2\,\langle \nabla V,\nabla f \rangle\,\e^{-f}\, . 
 \end{align} 
 \end{proof}

We specialize next to drift eigenfunctions, i.e., where $V = \lambda $ is constant.

 \begin{Lem}   \label{l:firstlowerD'}
 If $\cL u = -\lambda \, u $ on $\RR^n$, then 
 \begin{align}   \label{e:firstlowerD'}
D'(r)\geq  \frac{r}{2} \, D +  2\,\frac{U\,D}{r}  - 2 \, \e^{\frac{r^2}{4}} \, r^{1-n} \int_{B_r}\left( |\nabla u|^2 - \lambda u^2 \right) \, f\, \e^{-f} 
-2 \, \lambda \e^{\frac{r^2}{4}} \, r^{1-n} \int_{B_r}  u^2  \, \e^{-f}  \, ,
 \end{align}
 \end{Lem}
 
 \begin{proof}
Multiplying Lemma \ref{l:poho} by $r^{1-n}$ 
gives that
   \begin{align}	
 	2 \, r^{2-n} \int_{\partial B_r} u_r^2 - r^{2-n} \int_{\partial B_r} \left( |\nabla u|^2 - \lambda u^2 \right) &=(2-n)\, \frac{D}{r} +
	2 \, \e^{\frac{r^2}{4}} \, r^{1-n} \int_{B_r}\left( |\nabla u|^2 -\lambda u^2 \right) \, f\, \e^{-f} \notag \\
	&+ 2\lambda  \, \e^{\frac{r^2}{4}} \, r^{1-n} \int_{B_r}  u^2  \, \e^{-f} \, .
 \end{align}
 Using this in the formula for $D'$ from Lemma \ref{l:IandD}  gives
  \begin{align}
D'(r)&=  \frac{2-n}{r} \, D + \frac{r}{2} \, D + r^{2-n} \,   \int_{\partial B_r}  \left( |\nabla u|^2 - \lambda u^2 \right) \notag \\
&=  \frac{r}{2} \, D +  2 \, r^{2-n} \int_{\partial B_r} u_r^2 - 2 \, \e^{\frac{r^2}{4}} \, r^{1-n} \int_{B_r}\left( |\nabla u|^2 - \lambda u^2 \right)  f\, \e^{-f} 
-2 \, \lambda \e^{\frac{r^2}{4}} \, r^{1-n} \int_{B_r}  u^2  \, \e^{-f}   \, .
\end{align}
The lemma follows from this since $  r^{2-n} \int_{\partial B_r} u_r^2  \geq  \frac{ UD}{r}$ by   \eqr{e:CS}.  \end{proof}
  
  The next corollary  shows that $U$ is monotone for drift-harmonic functions.
  
  \begin{Cor}
 If $\cL u = 0$ on $\RR^n$, then 
    \begin{align}
 (\log U)' & \geq  
  \frac{2\, \int_{B_r} |\nabla u|^2 \, \left( \frac{r^2}{4} - f \right)\, \e^{-f} }{ r\, \int_{B_r} |\nabla u|^2 \,  \e^{-f}} \geq 0 
  \, .
 \end{align}
 \end{Cor}
 
 \begin{proof}
 Dividing by $D$ in \eqr{e:firstlowerD'} with $\lambda =0$, we see that
   \begin{align}
 (\log U)' & \geq  \frac{r}{2}    -   \frac{ 2\, \int_{B_r} |\nabla u|^2 \, f\, \e^{-f} }{r\, \int_{B_r} |\nabla u|^2 \,  \e^{-f}} = 
  \frac{2\, \int_{B_r} |\nabla u|^2 \, \left( \frac{r^2}{4} - f \right)\, \e^{-f} }{ r\, \int_{B_r} |\nabla u|^2 \,  \e^{-f}} \geq 0 
  \, .
 \end{align}
 \end{proof}
  
 When $u$ is not drift harmonic, then we will need to rewrite the right hand side of equation \eqr{e:firstlowerD'}.   This is done next (we record the result for a general $V$).
   
 \begin{Lem}    \label{l:combine}
 If $\cL\,u+V\,u=0$ on $\RR^n$,  then 
 \begin{align}
 \e^{\frac{r^2}{4}}  r^{1-n} \int_{B_r}\left( |\nabla u|^2 - V\,u^2 \right) \, f\, \e^{-f} =\frac{r}{4} \, \left(D -I\right)+\frac{1}{2} \, \e^{\frac{r^2}{4}}  r^{1-n} \int_{B_r}u^2 \, \left( \frac{n}{2} - f \right) \e^{-f}\, .
 \label{e:fromrem}
 \end{align}
 \end{Lem}
 
 \begin{proof}
 Observe first that since $\cL u = - V\,u$,  $\frac{r^{3-n}}{4} \,   \int_{\partial B_r} u u_r = \frac{r}{4} \, D $, and 
 \begin{align}
 \dv_f \, \left(   u \, f \, \nabla u
 \right)= \left( |\nabla u|^2 - V \, u^2 \right) \, f+ u\, \langle \nabla u , \nabla f \rangle   \, , 
 \end{align}
 we have
 \begin{align}
 \e^{\frac{r^2}{4}}  r^{1-n} \int_{B_r} \left( |\nabla u|^2 - V\, u^2 \right) \, f\, \e^{-f} & = \frac{r}{4} \, D -  \e^{\frac{r^2}{4}}  r^{1-n} \int_{B_r}u \langle  \nabla  u , \nabla f \rangle \,  \e^{-f} \, .	
 \end{align}
 Next since $\dv_f (u^2 \nabla f) = 2 u \langle \nabla u , \nabla f \rangle + u^2 \, \cL f = 2 u \langle \nabla u , \nabla f \rangle + u^2 \, \left( \frac{n}{2} - f \right) $, we have
 \begin{align}
 	\e^{\frac{r^2}{4}}  r^{1-n} \int_{B_r}u \langle  \nabla  u , \nabla f \rangle \,  \e^{-f} = \frac{r}{4} \, I - \frac{1}{2} \, \e^{\frac{r^2}{4}}  r^{1-n} \int_{B_r}u^2 \, \left( \frac{n}{2} - f \right) \e^{-f}	\label{e:abo} 
		\, .
\end{align}
Combining these two equations gives the claim.   
 \end{proof}

 As a corollary, we get a lower bound for $U'$.
 
 \begin{Cor}   \label{c:Upbound}
 If $\cL\,u+\lambda\,u=0$ on $\RR^n$,  then
 \begin{align}
 U'(r)\geq \frac{r}{2}+I^{-1}(r)\,\e^{\frac{r^2}{4}}\,  r^{1-n} \int_{B_r}u^2 \, \left( f-\frac{n}{2} -2\lambda  \right) \e^{-f}   \,  .
 \end{align}
 Furthermore, given $\delta > 0$, there exists $r_1$ so that if $U(\bar{r}) \geq \delta + 2|\lambda|$ for some $\bar{r} \geq r_1$, then there exists $R$ so that for all $r\geq R$
  \begin{align}
 U'(r)\geq \frac{r}{2}   \,  .
 \end{align}
 \end{Cor}
 
 \begin{proof}
 Combining Lemmas \ref{l:firstlowerD'} and \ref{l:combine} gives
 \begin{align}	 
	D'(r)\geq    2\, \frac{D\,U}{r} +   \frac{r}{2}\,I +\e^{\frac{r^2}{4}}\,  r^{1-n} \int_{B_r}u^2 \, \left( f-\frac{n}{2}  -2\lambda\right) \e^{-f}  \, .
 \end{align}
 The first claim follows from this  since $U'=\left(\frac{D'}{D}-\frac{2U}{r}\right)\,U$.
 
 To prove the second claim, we just need to show that there is some $R^2 \geq 2n + 8\lambda$ with
  \begin{align}
   \int_{B_R}u^2 \, \left( f-\frac{n}{2} -2\lambda  \right) \e^{-f}  \geq 0 \,  .
 \end{align}
 This follows immediately  since $I(r) \, \e^{ - \frac{r^2}{4} }$ grows rapidly by Theorem \ref{t:main}.
 \end{proof}

 \begin{proof}
 (of Theorem \ref{t:mainLem2}.)   
 We can assume that $\limsup_{r \to \infty} U(r) > 2|\lambda|$.  Thus,   the second part of 
  Corollary \ref{c:Upbound} applies and $U'(t)\geq \frac{t}{2}$ for $t>r_0$.  In particular,  $W(t) = U(t)-\frac{t^2}{4}$ satisfies $W'\geq 0$ for $t \geq r_0$.  
  After possibly increasing $r_0$, we can assume that $r_0^2 > 2\,n+8\,\lambda$ and, moreover, that $W(r_0) > 0$ (using Theorem \ref{t:main}).
  
    By Lemma \ref{l:IandD}, for $r>s>r_0$ 
 \begin{align}
 \log \frac{ I(s)}{I(r)} =    -2\int_s^r\frac{U}{t}\,dt  \geq    \frac{s^2-r^2}{4}-2\,W(r)\int_s^r\frac{1}{t}\,dt  = \frac{s^2-r^2}{4} \log \left(\frac{s}{r}\right)^{2\,W(r)} \, .
 \end{align}
 It follows that for any constant $c \leq r_0^2$
 \begin{align}
 \e^{\frac{r^2}{4}}\,\frac{r^{1-n}}{I(r)}\int_{r_0}^r &\left( s^{2} - c\right) \, s^{n-1}  \,I(s)\,\e^{-\frac{s^2}{4}}\,ds \geq  r^{1-n-2\,W(r)}\int_{r_0}^r \left( s^{n+1+2\,W(r)} -
 c\, s^{n-1+2\,W(r)}\right) \,ds \notag \\
 &\qquad =\frac{r^3}{n+2+2\,W(r)} - \frac{c\,r}{n+ 2\, W(r)} + \frac{1}{n+2+2\,W(r)}O(r^{1-n-2W(r)}) \, , 
 \end{align}
 where $O(r^{1-n-2W(r)})$ is a term that is bounded by a constant (depending on $r_0$) times $r^{1-n-2W(r)}$.
 Inserting this in Corollary \ref{c:Upbound} with $c=2\,n+8\,\lambda$ gives the claim.  
 \end{proof}

  \section{Drift harmonic functions on open manifolds}	\label{s:s3}
  
  In this section, we will show a natural generalization (Theorem \ref{t:openintro}) of \eqr{e:inductive} to open manifolds with nonnegative Ricci curvature and Euclidean volume growth.  In fact, the assumptions on the Ricci curvature and volume growth are only used to show that the function $b$ defined below is proper.  
  
  Let again $f$ be a function on $(0,\infty)$ with $f'\geq \frac{r}{2}$.  Suppose that $M$ is an open manifold, $b:M\to \RR$ is a proper function.  For a function $u:M\to \RR$, define  (cf.   \cite{CM1} and \cite{CM2})
  \begin{align}	\label{e:Ib}
	  I(r)&=r^{1-n}\int_{b=r}u^2\,|\nabla b|\, ,\\
  	D(r)&=r^{2-n}\,\e^{f(r)}\int_{b\leq r}|\nabla u|^2\,\e^{-f(b)}\, ,  \label{e:Db} \\
 	 U(r)&=\frac{D(r)}{I(r)}\, .	\label{e:Ub}
  \end{align}
   We set $\cL_{f}\,u=\Delta\,u-\langle \nabla u,\nabla f(b)\rangle$.  
 It follows that
  \begin{align}
  I'(r)&=r^{1-n}\int_{b=r}\frac{\nabla b}{|\nabla b|}u^2+\int_{b=r}u^2\,\frac{\nabla b}{|\nabla b|^2}\left(r^{1-n}\,|\nabla b|\,d\,\Vol\right)\notag\\
  &=r^{1-n}\,\e^{f(r)}\int_{b\leq r}\cL_f \,u^2\,\e^{-f(b)}+\int_{b=r}u^2\,\frac{\nabla b}{|\nabla b|^2}\left(r^{1-n}\,|\nabla b|\,d\,\Vol\right)\, ,
  \end{align}
  where $d\,\Vol$ is the volume element of the level set of $b$.  The co-area formula gives
  \begin{align}
 	 D'(r)&=\frac{2-n}{r}\,D+ f'(r) \, D+r^{2-n}\int_{b=r}\frac{|\nabla u|^2}{|\nabla b|}\, .
  \end{align}

 If $\cL_f\,u=0$, then $\cL_f \,u^2=2\,|\nabla u|^2$.    Therefore
\begin{align}
D(r)&=  \frac{1}{2}\,r^{2-n}\,\e^{f(r)}\int_{b\leq r}\cL_f\,u^2\,\e^{-f(b)}=r^{2-n}\,\int_{b= r}u\, \langle \nabla u , \frac{\nabla b}{|\nabla b|} \rangle \, .
  \end{align} 
  The Cauchy-Schwarz inequality (cf. \eqr{e:CS}) gives for $\cL_f\,u=0$
 \begin{align}	\label{e:CSgeneral}
 	\frac{D^2}{r} = r^{3-2n} \, \left(\int_{b= r} u u_r \right)^2 \leq    I \, r^{2-n} \int_{b=r} \frac{|\nabla u|^2}{|\nabla b|} \, .
\end{align}
   It follows that for $\cL_f \,u=0$
 \begin{align}	 	\label{e:boundmfld}
D'(r)=  \frac{2-n}{r} \, D + f' \, D + r^{2-n} \,   \int_{b=r}  \frac{ |\nabla u|^2 }{|\nabla b|}   \geq  \frac{2-n}{r} \, D + f' \, D + \frac{U}{r} \, D  \, .
 \end{align}
If   $\cL_f\,u=0$ and 
\begin{align}   \label{e:crucialeq}
\frac{\nabla b}{|\nabla b|^2}\left(r^{1-n}\,|\nabla b|\,d\,\Vol\right)= 0\, ,
\end{align}
 then $I'= 2\,\frac{D}{r}$ and $\left(\log I\right)'=\frac{2\,U}{r}$.  
Hence, by \eqr{e:boundmfld}
\begin{align}
P_{f,0}\,U\geq 0 \, .  \label{e:keyineq}
\end{align}

By \cite{CM1} if $b^{2-n}$ is harmonic, then \eqr{e:crucialeq} holds.    This is due to the following:  

\begin{Lem}
Let $d\,\Vol$ denote the volume element of the level set of a function $v$, then
\begin{align}
\frac{\nabla v}{|\nabla v|^2}\,\left(|\nabla v|\,d\,\Vol\right)=\frac{\Delta\,v}{|\nabla v|}\,d\,\Vol\, .
\end{align}
\end{Lem}

\begin{proof}
An easy calculation shows that the change in volume element (of the level set) is
\begin{align}
\dv\,\left(\frac{\nabla v}{|\nabla v|^2}\right)-\langle\nabla_{\frac{\nabla v}{|\nabla v|}}\left(\frac{\nabla v}{|\nabla v|^2}\right),\frac{\nabla v}{|\nabla v|}\rangle=\frac{\Delta\,v}{|\nabla v|^2}-\frac{\nabla v (|\nabla v|)}{|\nabla v|^3}\, .
\end{align}
From this the claim follows.
\end{proof}

It follows from \eqr{e:keyineq} together with Theorem \ref{t:generalbound} that:

\begin{Thm}  \label{t:open}
Suppose that $f'\geq \frac{r}{2}$.  Given $\epsilon>0$ and $\delta>0$, if $\Delta\,b^{2-n}=0$, then there exist $r_1>0$ so that if $U(\bar{r}) \geq \delta$ for some $\bar{r} \geq r_1$ and $\cL_f u = 0$, then   for all $r\geq R (\bar{r})$
\begin{align} 
U(r)> \frac{1}{2}\,r^2-n-\epsilon\, .
\end{align}
\end{Thm}

In particular, it follows from \cite{CM1} that if $M$ is an open manifold with nonnegative Ricci curvature and Euclidean volume growth and $b$ is given by $b^{2-n}=G$, then $b$ is proper and thus the conclusion of Theorem \ref{t:open} holds giving Theorem \ref{t:openintro}.    (Here $G$ is the Green's function.)

 \section{Approximation of eigenfunctions}
 
Theorem \ref{t:main} implies that if $\cL u = - \lambda u$ on $\RR^n$, then either $u$ grows at most  polynomially or at least as fast at $r^{-p} \e^{ \frac{r^2}{4} }$  for some power $p$.
In the first case, $\| u \|_{L^2} < \infty$, so $u$ is a polynomial and $\lambda$ a half-integer.
The next theorem gives a local version of this; we will see a more general version of this in the next section.

\begin{Thm}	\label{t:appt}
Given $k \in \ZZ$ and $R_0$, there exist $C$ and $R_1$ so that if
  $\cL u = - \frac{k}{2}$ on $B_R$  for some $R \geq R_1$, then
 there is a polynomial $v$ of degree at most $k$ so that
 \begin{align}
 	\sup_{B_{R_0}} |u-v|^2 \leq C\, R^{4n-1+\max \{ 0 , 2k+2\}} \, \, \e^{ - \frac{R^2}{2} } \int_{ B_{R+ \frac{1}{R}} \setminus B_{R-  \frac{1}{R}}}   u^2  \, .
 \end{align}
\end{Thm}

\begin{proof}
We will prove this in two steps.  Suppose first that $k \leq -1$.  Lemma \ref{l:Uonetime} gives
\begin{align}	\label{e:fromU1}
	(\log U)' \geq \frac{2-n}{r} + \frac{r}{2} + \frac{r}{2U} - \frac{U}{r} \, .
\end{align}
We will show first that $U$ goes above $n$ on any interval $[r_0 , r_0 + 1]$ for $r_0 \geq 2n$.  To see this, suppose that $U \leq n$ on such an interval and use \eqr{e:fromU1} 
to get that
\begin{align}
	U' \geq \frac{r}{2} + U \, \left(\frac{2-n}{r} + \frac{r}{2} - \frac{n}{r}   \right) > n \, .
\end{align}
This is impossible since $0 \leq U \leq n$, giving the claim.  Thus,  Theorem \ref{t:main} gives $\bar{R}$ depending on $n$ so that $U(r) >  \frac{r^2}{2} - n$ for all $r \geq \bar{R}$.
Given  $r \geq \bar{R}$, 
integrating this from $r$ to $R$ gives  
\begin{align}
	\log \frac{ I(R)}{I(r)} \geq 2 \, \int_r^R \left( \frac{s}{2} - \frac{n}{s} \right) \, ds = \frac{1}{2} \, \left( R^2 - r^2 \right) - 2n \, \log  \frac{R}{r} \, .
\end{align}
Letting $r = \min \{ \bar{R} , 2 R_0 \}$, exponentiating and applying elliptic estimates  gives
\begin{align}	\label{e:4p6}
	\sup_{B_{R_0}} |u|^2 \leq c \, I(r) \leq C \, R^{2n} \, \e^{ - \frac{R^2}{2} } \, I(R)    \, .
\end{align}
The case $k\leq -1$ follows from this since $I(R) \leq c \, R^{2-n} \, \int_{ B_{R+ \frac{1}{R}} \setminus B_R} u^2$.

Suppose now that $k \geq 0$ and let $w$ be any $(k+1)$-st partial derivative of $u$.  It follows that $\cL w = - \frac{1}{2} \, w$, so \eqr{e:4p6} implies that
\begin{align}
	\sup_{B_{R_0}} |\nabla^{k+1}u|^2 \leq  C \, R^{2n} \, \e^{ - \frac{R^2}{2} } \int_{\partial B_R} |\nabla^{k+1} u|^2 \, .
\end{align}
Elliptic estimates on balls of radius $R^{-1}$ centered on $\partial B_R$ give that
\begin{align}
	\sup_{\partial B_R} |\nabla^{k+1} u|^2 \leq C \, R^{2k+2 +n} \, \int_{ B_{R+ \frac{1}{R}} \setminus B_{R- \frac{1}{R}}} u^2
\end{align}
The theorem follows  with $v$ given by the degree $k$ Taylor polynomial for $u$ at $0$.
\end{proof}

\section{Approximate eigenfunctions on cylinders}
   
  In this section, we let $M  = N  \times \RR^{n}$ be a product manifold where $N$ is closed.  Let $x$ be coordinates on $\RR^{n}$,  define $f = \frac{|x|^2}{4}$   and  the drift Laplacian 
 $
  	\cL = \Delta - \frac{1}{2} \, \nabla_x = \Delta_N + \cL_{\RR^n}  
$.
  Given a function $u$, we define $I$ and $D$ by
  \begin{align}
  	I(r) &= r^{1-n} \int_{|x| = r} u^2 \, , \\
	D(r) &= r^{2-n} \, \int_{|x| = r} u \, u_r = \e^{ \frac{r^2}{4}} \, r^{2-n} \, \int_{|x| < r } \left(|\nabla u|^2 + u \cL \, u \right) \, \e^{-f} \, .
  \end{align}
  Here $u_r$ denotes the normal derivative of $u$ on the level set $|x|  = r$.  Since $N$ is compact, $f$ is proper and the integrals exist.
    It is easy to see that $I' = \frac{2D}{r}$ and $(\log I)' = \frac{2U}{r}$, where 
  the frequency $U$ is given by $U = \frac{D}{I}$.
   
  \vskip1mm
The next theorem gives a strong approximation for approximate eigenfunctions on $M$.  The theorem is stated for eigenvalue $- \frac{1}{2}$ for simplicity, but can be modified easily for other eigenvalues by arguing as in the previous section.  This result is a key ingredient in \cite{CM3}.

\begin{Thm}	\label{t:app}
There exist $\bar{R}$ and $C$ depending on $n$ so that if
 $v$  is a function on $\{ |x| \leq R \}$, where $\bar{R} \leq R   $,     and 
\begin{enumerate}
\item $\left| - \frac{1}{2} \, v^2 + v \, \cL v \right| \leq \psi^2 + \epsilon\, \left( \frac{v^2}{2} + |\nabla v|^2 \right) $, where $\psi$ is a function and $\epsilon < \frac{1}{2}$,
\end{enumerate}
then we get for any $\Lambda \in (0, 1/2)$ that
\begin{align}	\label{e:goal}
	\int_{|x| < {4n}} v^2 \, \e^{-f}  \leq  \frac{2}{\Lambda} \, \| \psi \|_{L^2}^2 + C \, I(R) \, R^{2n} \, \e^{ - \frac{ (1-\epsilon-\Lambda)\, R^2}{2(1+\epsilon+ \Lambda)^2}}  \, .
\end{align}
\end{Thm}

\vskip2mm
In the proof, we will need a modified version of the frequency.  Define  $E(r) $ by 
\begin{align}
	E(r)  = 
	 r^{2-n} \, \e^{ \frac{r^2}{4} } \, \int_{|x| <r} \left\{ |\nabla v|^2 + \frac{1}{2} \, v^2   \right\} \e^{-f}= D(r) - r^{2-n} \, \e^{ \frac{r^2}{4} } \, \int_{|x|<r}   \left( v\cL v - \frac{1}{2} \, v^2   \right)    \e^{-f}  \, . \notag
\end{align}
  We   define a modified frequency $U_E$ by
\begin{align}
	U_E(r) = \frac{E(r)}{I(r)} \, .
\end{align}

\begin{Lem}	\label{l:diffineq}
We have
\begin{align}
	\left( \log U_E \right)'    \geq 
	\frac{2-n}{r}   + \frac{r}{2}   +   \frac{r}{2\, U_E} +  \frac{U}{r} \left(  \frac{D}{E} - 2  \right) \, .
\end{align}
\end{Lem}

\begin{proof}
Differentiating   gives that
\begin{align}
	E'(r) = \frac{2-n}{r} \, E + \frac{r}{2} \, E + \frac{r}{2}\, I +  r^{2-n} \, \int_{|x|=r}   |\nabla v|^2  \geq 
	\frac{2-n}{r} \, E + \frac{r}{2} \, E + \frac{r}{2}\, I +  \frac{UD}{r} \, , 
\end{align}
where the inequality used  the Cauchy-Schwarz inequality, \eqr{e:CS}.  The lemma follows from this since  $\frac{I'}{I} = \frac{2U}{r}$.
\end{proof}

\begin{proof}[Proof of Theorem \ref{t:app}]
We get \eqr{e:goal} immediately if
\begin{align}	\label{e:case1}
	\int_{|x|<{4n} } v^2 \e^{-f} < \frac{2}{\Lambda} \, \| \psi \|_{L^2}^2 \, .
\end{align}
 Suppose, therefore, that \eqr{e:case1} fails.  Given any $r \geq 4n$, it follows from (1)   that
\begin{align}
 	\left| D - E \right| & \leq   \epsilon \, E +  r^{2-n}  \,   \e^{ \frac{r^2}{4} } \, \| \psi \|_{L^2}^2 \leq 
	 \epsilon \, E +  \Lambda \, r^{2-n}  \,   \e^{ \frac{r^2}{4} } \, \int_{ |x| < 4n} \frac{v^2}{2} \, \e^{-f} \leq (\epsilon + \Lambda) \, E 
	 \, .
\end{align}
Therefore, if  $4n \leq r  $, then:  $0 \leq I'(r) $,
\begin{align}  \label{e:step1a}
	\left| U - U_E \right| (r) &\leq ( \epsilon + \Lambda) \, U_E (r) \, , \\
	\left( \log U_E \right)'    &\geq 
	\frac{2-n}{r}   + \frac{r}{2}   +   \frac{r}{2\, U_E} - ( 1+ \epsilon + \Lambda)^2 \, \frac{U_E}{r} \, , \label{e:step1}
\end{align}
where the last inequality also used Lemma \ref{l:diffineq}.

We first  show that      $ \max_{ [4n,8n]} \, U_E \geq n$.  To see this, suppose instead that $U_E < n$ on $[4n,8n]$ and  use \eqr{e:step1} to get 
\begin{align}	 
	\left( \log U_E \right)'      > \frac{2n}{U_E} \, .
\end{align}
Multiplying by $U_E$, we  get an interval of length $4n$ where $0 < U_E < n$ but $2n < U_E'$.  This is impossible, so we conclude that
 $ \max_{ [4n,8n]} \, U_E \geq n$ as claimed.

We claim that there exists $\bar{R}= \bar{R} (n) \geq 5n$ so that for all $r \geq \bar{R}$ we have 
\begin{align}	\label{e:thresh}
	U_E (r) > \frac{ r^2 -2n}{2(1+\epsilon +\Lambda)^2} \, .
\end{align}
The key  is that if \eqr{e:thresh} fails for some $r \geq 4n$, then \eqr{e:step1} implies that
\begin{align}	 	\label{e:maxpU}
	\left( \log U_E \right)'    \geq 
	\frac{2}{r}     +   \frac{r}{2\, U_E} \geq 
	\frac{6}{r}         \, .
\end{align}
On the other hand, for $r \geq 4n$, we have
\begin{align}	\label{e:loggb}
	\left( \log \frac{ r^2 -2n}{2(1+\epsilon +\Lambda)^2}  \right)' = \frac{2r}{r^2-2n} < \frac{3}{r} 
	\, .
\end{align}
Integrating \eqr{e:maxpU} and \eqr{e:loggb} and using  that $ \max_{ [4n,8n]} \, U_E \geq n$, gives an upper bound for the maximal interval where \eqr{e:thresh} fails. The first derivative test, \eqr{e:maxpU}, and \eqr{e:loggb}  imply that once \eqr{e:thresh} holds for some $R \geq 4n$, then it also holds for all   $r \geq R$.  This gives the claim.

Using \eqr{e:step1a} and \eqr{e:thresh}, we get for $r \geq \bar{R}$ that
\begin{align}	\label{e:thresh2}
	U (r) \geq  (1-\epsilon-\Lambda) U_E (r) >    \frac{ (1-\epsilon-\Lambda)  }{(1+\epsilon+ \Lambda)^2} \left( \frac{r^2}{2} - n \right) \equiv \kappa \,  \left( \frac{r^2}{2} - n \right) \, ,
\end{align}
where the last equality defines $\kappa$.
Integrating this from $\bar{R}$ to $R$ gives that
\begin{align}
	\log \frac{I(R)}{I(\bar{R})} \geq 2\, \int_{\bar{R}}^R  \frac{U(r)}{r} \, dr &\geq \kappa \int_{\bar{R}}^R \left(    \, r - \frac{2n}{r} \right) \, dr  =  \kappa \, \left( \frac{R^2 - \bar{R}^2}{2} -  2n  \, \log \frac{R}{\bar{R}} \right) \, .
\end{align}
Since    $\bar{R}$ is uniformly bounded, exponentiating gives that
\begin{align}	\label{e:521}
	\sup_{4n \leq r \leq \bar{R} } \, \, I(r) = I(\bar{R}) \leq c_n \, I(R) \, R^{2n\, \kappa} \, \e^{ - \frac{\kappa}{2}\, R^2 }   \, .
\end{align}

We use the reverse Poincar\'e to get   the  integral bound on $|x| < 4n$.  Let $\eta \leq 1$ be a cutoff that is one on $\{ |x| < 4n \}$,  zero for $|x| > 5n$, and has $|\nabla \eta| \leq 1$. 
Integration by parts   gives
\begin{align}
	\int \eta^2 \left( |\nabla v|^2 + \frac{v^2}{2} \right) \e^{-f} &=- \int \left( 2\eta v \langle \nabla v , \nabla \eta \rangle + \eta^2 \left(v \cL v - \frac{v^2}{2} \right) \right) \, \e^{-f} \, .
\end{align}
Using (2) on the last term (note that $\epsilon < 1/2$) and absorbing the first term  gives
\begin{align}
	\frac{1}{2} \, \int \eta^2 \left( |\nabla v|^2 + \frac{v^2}{2} \right) \e^{-f} &=\| \psi \|_{L^2}^2 - \int \left( 2\eta v \langle \nabla v , \nabla \eta \rangle   \right) \, \e^{-f} \notag \\
	&\leq \| \psi \|_{L^2}^2  + \frac{1}{2} \, \int  \eta^2 v |\nabla v|^2 \, \e^{-f} + 2\int |\nabla \eta|^2    v^2 \, \e^{-f}   \, .
\end{align}  
Since $\eta = 1$ for $|x| < 4n$ and $|\nabla \eta| \leq 1$ is only nonzero for $4n < |x| < 5n$, it follows that
\begin{align}	\label{e:524}
	\int_{ \{ |x| < 4n \} } v^2 \, \e^{-f}   \leq 4\, \| \psi \|_{L^2}^2 + 8 \, \int_{\{ 4n < |x| < 5n \} } v^2 \, \e^{-f} \leq 4\, \| \psi \|_{L^2}^2 + C \, I(5n) \, ,
\end{align}
where we used that $I'(r) \geq 0$  for $r \geq 4n$.
Combining  \eqr{e:521} and \eqr{e:524} gives \eqr{e:goal}.
\end{proof}

\end{document}